\documentclass[a4paper,reqno,twoside,11pt]{bimsart}
\RequirePackage{amsmath,amsthm,amssymb,bims}

\RequirePackage{hyperref}

\newcommand\issuenumber{96}
\renewcommand\month{Winter}
\renewcommand\year{2025}
\newcommand\firstpage{1}
\newcommand\lastpage{?}
\newcommand\received{Received on 24-07-2025.}
\processdata%

\newcommand\thistitle{Elementary Proofs of Ring Commutativity Theorems} \newcommand\thistitleshort{Ring Commutativity Theorems}
\title[\thistitleshort]{\thistitle}

\newcommand\thisauthor{Michael Kinyon and Desmond MacHale}
\newcommand\thisauthorshort{Kinyon and MacHale}
\author[\thisauthorshort]{\thisauthor}

\address[Kinyon]{%
Department of Mathematics, University of Denver, Denver, CO 80122, USA
}

\address[MacHale]{%
School of Mathematical Sciences, University College Cork, Cork, T12 XF62, Ireland
}

\email[Kinyon]{michael.kinyon@du.edu}
\email[MacHale]{d.machale@ucc.ie}


\theoremstyle{plain}
\newtheorem*{jac}{Jacobson's Theorem}
\newtheorem*{her}{Herstein's Theorem}

\newcommand{\End}[1]{\mathrm{End}(#1)}
\newcommand{\ad}[1]{\mathrm{ad}(#1)}
\newcommand{\setof}[2]{\{#1\,:\,#2\}}

\newcommand\authorbio{
{\noindent\textbf{Michael Kinyon}} is Professor of Mathematics at the University of Denver. His mathematical interests include quasigroup theory and semigroup theory, and he is especially keen on using automated deduction software to prove theorems.
\\
{\textbf{Desmond MacHale}} is Emeritus Professor of Mathematics at University College Cork where he taught for forty years.
His main mathematical interests are in groups and rings, but he has been known to dabble in number theory, Euclidean geometry, trigonometric inequalities, combinatorial geometry, problem posing and solving, and the
humour of mathematics. He has written several biographical books on George Boole.
}

\begin{document}


\begin{abstract}
Jacobson's commutativity theorem says that a ring is commutative if, for each $x$, $x^n = x$ for some $n > 1$. Herstein's generalization says that the condition can be weakened to $x^n-x$ being central. In both theorems, $n$ may depend on $x$. In this paper, in certain cases where $n$ is a fixed constant, we find equational proofs of each theorem. For the odd exponent cases $n = 2k+1$ of Jacobson's theorem, our main tool is a lemma stating that for each $x$, $x^k$ is central. For Herstein's theorem, we consider the cases $n=4$ and $n=8$, obtaining proofs with the assistance of the automated theorem prover \textsc{Prover9}.
\end{abstract}

\maketitle


\section{Introduction}
N. Jacobson's celebrated commutativity theorem for rings \cite{NJ1945} and its generalization by I. N. Herstein \cite{Herstein} state:

\begin{jac}
	Let $R$ be a ring such that for each $x\in R$, there exists an integer $n = n(x) > 1$ such that $x^n = x$. Then $R$ is commutative.
\end{jac}

\begin{her}
	Let $R$ be a ring such that for each $x\in R$, there exists an integer $n = n(x) > 1$ such that $[x^n - x,y] = 0$ for all $y\in R$. Then $R$ is commutative.
\end{her}

Jacobson's Theorem is a generalization of Wedderburn's ``Little'' Theorem that every finite division ring is commutative. Rings satisfying the hypothesis of Jacobson's Theorem have been given various names, such as \emph{potent} rings \cite{AndersonDanchev,Oman}, $J$-rings \cite{HuzSiv}, and probably others of which we are unaware. Choosing the first one, Jacobson's Theorem can be stated succinctly as \emph{potent rings are commutative}. Rings satisfying the hypothesis of Herstein's Theorem seem to have never been given a separate name as far as we know, probably because the hypothesis is both necessary and sufficient for a ring to be commutative.

As noted in the statements, both theorems allow the exponent $n$ of the power $x^n$ to depend on $x$. For Jacobson's, there are various proofs in the literature; perhaps the nicest was given independently by J. W. Wamsley \cite{Wamsley} and T. Nagahara and H. Tominaga \cite{NT}. 

In this paper we are interested in both theorems in the case where $n$ is a fixed integer not depending on $x$, that is, rings $R$ for which there exists an integer $n\geq 2$ such that $x^n = x$ for all $x\in R$. There have been various names suggested for such rings: $J$-rings \cite{Luh,Luh2} (also used, as noted, for potent rings \cite{HuzSiv}), Jacobson rings \cite{Dolan}, $n$-rings \cite{MB}, $n$-potent rings \cite{ADT}, and, again, probably others of which we are unaware. Since elements $x$ satisfying $x^n = x$ for some $n$ are often called $n$-potent elements, we choose the name $n$-\emph{potent} rings. We will primarily use the name in the interest of simplifying theorem statements.

The well known class of \emph{Boolean} rings coincides with what we are calling $2$-potent rings. Jacobson's Theorem for such rings is a standard exercise with an easy equational proof (see the first alternative proof to Theorem \ref{Thm:x2=x} below). Indeed, it is a consequence of Birkhoff's Completeness Theorem for Equational Logic \cite{Birkhoff} that for each fixed $n$, an equational proof of Jacobson's Theorem exists. 

Among the papers devoted to proofs of Jacobson's Theorem for particular fixed $n$ \cite{BuckMac,FM,Wavrik,Zhang}, we would particularly like to single out the \emph{tour de force} of Y. Morita \cite{Morita}, who gave human equational proofs for most even numbers $\leq 50$ and all odd numbers $\leq 25$. 

There is certainly similar interest in equational proofs of Herstein's Theorem for fixed $n$, although the literature is not as extensive; see, e.g., \cite{BuckMac2,MacHale}.

It is natural to try to use computer assistance to generate equational proofs of either theorem in the case of fixed $n$ \cite{PZ,Wavrik,Zhang}. For Jacobson's Theorem, M. Brandenburg \cite{MB} has recently done some very exciting work along these lines. 

This paper is in two parts. In {\S}2, we discuss Jacobson's Theorem for certain cases of small $n$. For $n$ odd, our main tool is a useful result we think is new (Lemma \ref{Lem:buckley}): \emph{if $R$ is a $(2k+1)$-potent ring, then for all $x\in R$, $x^k$ is central.} This result was first found by Stephen Buckley (unpublished); our proof differs from his. We illustrate how helpful the lemma is for $n=3,5,$ and $7$. We also give some (we believe) new proofs for various even $n$. All but one of the proofs in {\S}2 were originally human generated; the proof of Lemma \ref{Lem:idem_center} was first found by an automated theorem prover and subsequently humanized.

In {\S}3, we turn to Herstein's Theorem for the specific cases of $n=2$, $4$, and $8$. The proofs were found with the assistance of \textsf{Prover9}, the automated theorem prover developed by W. McCune \cite{McCune}. In fact, the beginning of the authors' collaboration was an email suggestion by the second author to the first that it would be interesting to find automated proofs of Herstein's Theorem for $n=4$ and $8$, and then to see if such proofs could be suitably humanized. 

We would judge the humanization effort to be quite successful for $n=4$ (Theorem \ref{Thm:fourth}) and somewhat successful for $n=8$ (Theorem \ref{Thm:exp8}). For the latter proof, although it is certainly possible for a patient human to follow the individual steps, the overall pattern is difficult to see. We have no idea how or even if a general idea can be extracted from the proof which could be applied to higher powers of $2$.

We conclude this introduction by discussing conventions. Rings are assumed to be associative but not assumed to have a unity, that is, they are what some, following a suggestion of Jacobson (\cite{NJ1989}, pp. 155--156), would call a ``rng''.

The \emph{centre} of a ring $R$ is the subring $Z(R) = \setof{a\in R}{ar=ra,\,\forall r\in R}$. Elements of $Z(R)$ are said to be \emph{central}.

In our proofs, especially in {\S}3, we will make heavy use of the commutator $[x,y] = xy-yx$. This is an interesting binary operation in its own right, but here we mainly use it as a computational tool. It turns out that introducing commutators into \textsf{Prover9} input files and feeding the program basic facts about commutators helps quite a bit in finding proofs. We will discuss the commutator identities we need at the beginning of {\S}3.

\section{Jacobson's Theorem for $n$-potent rings}

Our main interest in this section is giving equational proofs that $n$-potent rings are commutative for various values of $n$. However, for some preliminary results, there is essentially no extra work involved in giving proofs for reduced rings. A ring $R$ is said to be \emph{reduced} if it has no nilpotent elements. This can be equivalently described by the condition
\begin{equation}\label{Eqn:reduced}
	x^2=0\implies x=0\,,
\end{equation}
for all $x\in R$. Every potent ring $R$ is reduced: if $x\in R$ satisfies $x^2 = 0$, let $n = n(x) > 1$ be such that $x^n = x$. Then $x = x^n = x^2 x^{n-1} = 0$. Every reduced ring is a subdirect product of domains \cite{AR,Klein}, but it would breach the spirit of this paper to use this rather deep structural result.

Every reduced ring $R$ satisfies the condition
\begin{equation}\label{Lem:zero_comm}
	xy = 0\implies yx = 0\,,
\end{equation}
for all $x\in R$. Indeed, if $xy=0$, then $(yx)^2 = y(xy)x = 0$ and so $yx = 0$ by \eqref{Eqn:reduced}. Rings satisfying \eqref{Lem:zero_comm} are said to be \emph{reversible} \cite{Cohn}.

Reduced rings are neither defined nor characterized by identities. Thus even elementary proofs in reduced rings unavoidably use \eqref{Eqn:reduced} and hence are not, strictly speaking, equational. However, it is straightforward to convert such proofs to equational ones in $n$-potent rings. For example, to prove directly that an $n$-potent ring ($n>1$) is reversible, note that if $xy=0$, then $yx = (yx)^n = y(xy)^{n-1}x = 0$. 

An idempotent $e$ (that is, an element satisfying $e^2=e$) of a reduced ring $R$ is central. This is well known and has a short, standard proof: check that $(ex - exe)^2=0$ and $(xe - exe)^2=0$, so $R$ being reduced implies $ex - exe = 0$ and $xe - exe = 0$, hence $ex=exe=xe$. In fact, the same expressions occur in a mild generalization.

\begin{lemma}\label{Lem:rev_idem}
	In reversible rings, idempotents are central.
\end{lemma}
\begin{proof}
	If $e$ is an idempotent in a ring $R$, then for all $x\in R$, $e(x-ex) = 0$ and $(x-xe)e = 0$. If $R$ is reversible, then \eqref{Lem:zero_comm} implies $(x-ex)e = xe-exe = 0$ and $e(x-xe) = ex-exe = 0$. Thus $xe=exe=ex$, that is, $e$ is central.
\end{proof}

In the reduced case, the proof of the following useful generalization is only a bit more involved than the classic proof and is based on the same idea.

\begin{lemma}\label{Lem:idem_center}\label{Lem:ez}
	Let $R$ be a reduced ring. If $c\in R$ satisfies $c^2 = tc$ for some integer $t$, then $c$ is central.	
\end{lemma}
\begin{proof}
	Firstly, for all $x\in R$, $c^2 x c = tc\cdot xc = cx\cdot tc = cxc^2$. Thus $c[c,x]c = c(cx-xc)c = 0$. In particular, $(c[c,x])^2 = 0$ and 
	$([c,x]c)^2 = 0$. Since $R$ is reduced, \eqref{Eqn:reduced} yields both $c[c,x] = 0$ and $[c,x]c = 0$. These imply, respectively,
	\begin{align}
		xc[c,x] &= 0 \qquad \text{and}\label{Eqn:ic_1} \\
		[c,x]cx &= 0\,. \label{Eqn:ic_2}
	\end{align}
	Applying \eqref{Lem:zero_comm} to \eqref{Eqn:ic_1}, we get
	\begin{equation}\label{Eqn:ic_3}
		[c,x]xc = 0\,.
	\end{equation}
	Subtracting \eqref{Eqn:ic_3} from \eqref{Eqn:ic_2}, we obtain $[c,x]^2 = 0$. Since $R$ is reduced, $[c,x]=0$ for all $x\in R$, that is, $c\in Z(R)$.
\end{proof}

\begin{lemma}\label{Lem:red_n-1}\label{Lem:n-1}
	Let $R$ be a reduced ring and let $n>1$ be an integer. If $c\in R$ satisfies $c^n = c$, then $c^{n-1}$ is a central idempotent.
\end{lemma}
\begin{proof}
	If $n=2$, then $c^{n-1} = c$, while if $n > 2$, then $c^{n-1}c^{n-1} = c^n c^{n-2} = cc^{n-2} = c^{n-1}$. In either case, $c^{n-1}$ is an idempotent and so Lemma \ref{Lem:idem_center} applies.
\end{proof}

We now specialize from reduced rings to $n$-potent rings, our real interest.
The following result will turn out to be crucial, and can be viewed as improving Lemma \ref{Lem:n-1} in the case of odd $n$.

\begin{lemma}\label{Lem:buckley}
	Let $k\geq 1$ be an integer and let $R$ be a $(2k+1)$-potent ring. Then $x^k\in Z(R)$ for all $x\in R$.
\end{lemma}
\begin{proof}
	Firstly, $x^{2k}$ is a central idempotent by Lemma \ref{Lem:n-1}. Next we show $x^{3k} = x^k$. The claim is clear if $k=1$, while if $k > 1$, then $x^{3k} = x^{2k+1} x^{k-1} = xx^{k-1} = x^k$. Finally, $(x^{2k} + x^k)^2 = (x^{2k})^2 + 2x^{3k} + x^{2k} = 2(x^{2k} + x^k)$. By Lemma \ref{Lem:ez} (with $c=x^{2k}+x^k$, $t=2$), $x^{2k} + x^k\in Z(R)$ and so
	$x^k = (x^{2k}+x^k) - x^{2k}\in Z(R)$ for all $x\in R$, as claimed.
\end{proof}

For $n$ even, we have the following.

\begin{lemma}\label{Lem:char}
	Let $n>1$ be an even integer and let $R$ be an $n$-potent ring. Then $2x=0$ for all $x\in R$.
\end{lemma}
\begin{proof}
	For all $x\in R$, $-x = (-x)^n = x^n = x$. 
\end{proof}

For the remainder of this section, we establish commutativity theorems for $n$-potent rings for various $n$. We start with a classic, recalling that $2$-potent rings are the same as Boolean rings.

\begin{theorem}\label{Thm:x2=x}
	Every $2$-potent ring is commutative.
\end{theorem}
\begin{proof}
	This is the case $n=2$ of Lemma \ref{Lem:n-1}.
\end{proof}

As an alternative, here is the classic proof.

\begin{proof}[Alternative proof 1]
	Let $R$ be $2$-potent. For all $x,y\in R$, $x + y = (x+y)^2 = x^2 + xy + yx + y^2 = x + xy + yx + y$. Cancelling, we have $xy+yx=0$, and so $xy = -yx = yx$ by Lemma \ref{Lem:char}.
\end{proof}

We also present here the stunning proof of Brandenburg \cite{MB}.

\begin{proof}[Alternative proof 2] For all $x,y$ in a ring $R$, $xy-yx=$
	\[
	=((x + y)^2 - (x + y)) - (x^2 - x) - (y^2 - y) + ((yx)^2 - yx) - ( (-yx)^2 - (-yx))\,.
	\]
	If $R$ is $2$-potent, then the right hand side equals $0$.
\end{proof}

For a plethora of proofs and variations of our next result, see \cite{BuckMac}. 

\begin{theorem}\label{thm:x3=x}
	Every $3$-potent ring is commutative.
\end{theorem}
\begin{proof}
	This is the case $k=1$ of Lemma \ref{Lem:buckley}.
\end{proof}

\begin{proof}[Alternative proof]
	Let $R$ be $3$-potent. By Lemma \ref{Lem:n-1}, $x^2$ is central for all $x\in R$, so $2x = 2x^3 = (x^2+x)^2 - x^4 - x^2$ is central. Next,
	\[
	x^2 + x = (x^2 + x)^3 = x^6 + 3x^5 + 3x^4 + x^3 =
	x^2 + 3x + 3x^2 + x = 2\cdot 2(x^2 + x)\in Z(R)\,.
	\]
	Finally, $x = (x^2+x)-x^2\in Z(R)$ for all $x\in R$, that is, $R$ is commutative.
\end{proof}

\begin{lemma}\label{Lem:ab+ba}
	Let $k$ be an integer and let $R$ be a ring in which $k(xy+yx)\in Z(R)$ for all $x,y\in R$. Then $kx^2\in Z(R)$ for all $x\in R$.
\end{lemma}
\begin{proof}
	We have $k x^2\cdot y + k xyx = x\cdot k(xy+yx) = k(xy+yx)\cdot x = kxyx + y\cdot kx^2$ and so the desired result follows from canceling $kxyx$.
\end{proof}

\begin{lemma}\label{Lem:x2+x}
	If $R$ is a ring in which $x^2 + x\in Z(R)$ for all $x\in R$, then $R$ is commutative.
\end{lemma}
\begin{proof}
	For all $x,y\in R$, we have $(x+y)^2 + x + y\in Z(R)$, that is, $(x^2 + x) + (y^2+y) + xy + yx\in Z(R)$. It follows that $xy+yx\in Z(R)$, and thus $x^2\in Z(R)$ by Lemma \ref{Lem:ab+ba} (with $k=1$). Therefore $x = (x^2 + x) - x^2\in Z(R)$ for all $x\in R$, i.e., $R$ is commutative.
\end{proof}

\begin{lemma}\label{Lem:binom}
	If $n$ is a power of $2$, then each binomial coefficient $\binom{n}{k}$ is even except for $k=0$ and $k=n$.
\end{lemma}
\begin{proof}
	This is easily established by induction. 
\end{proof}

\begin{theorem}\label{Thm:x4=x}
	Every $4$-potent ring is commutative.
\end{theorem}
\begin{proof}
	Let $R$ be $4$-potent.
	By Lemma \ref{Lem:char}, $2x=0$ for all $x\in R$. Thus $(x^2+x)^2 = x^4 + x^2 = x^2 + x$ using Lemma \ref{Lem:binom}. By Lemma \ref{Lem:ez}, $x^2 + x$ is central for all $x\in R$, and so $R$ is commutative by Lemma \ref{Lem:x2+x}.
\end{proof}

\begin{theorem}\label{Thm:x5=x}
	Every $5$-potent ring is commutative.
\end{theorem}
\begin{proof}
	Let $R$ be $5$-potent. By Lemma \ref{Lem:buckley}, $x^2\in Z(R)$ for all $x\in R$. Thus $(x^2+x)^2 - x^4 - x^2 = 2x^3$ is central for all $x\in R$.
	Next $2x = 2x^5 = 2x^3\cdot x^2\in Z(R)$ for all $x$. Finally,
	\begin{align*}
		x^2 + x &= (x^2 + x)^5 \\
		&= x^{10} + 5x^9 + 10x^8 + 10x^7 + 5x^6 + x^5 \\
		&= x^2 + 5x + 10x^4 + 10x^3 + 5x^2 + x \\
		&= 2(5x^4 + 5x^3 + 3x^2 + 3x)\in Z(R)\,.
	\end{align*}
	By Lemma \ref{Lem:x2+x}, $R$ is commutative.
\end{proof}

We will postpone the case $n=6$ for now.

\begin{theorem}\label{Thm:x7=x}
	Every $7$-potent ring is commutative.
\end{theorem}
\begin{proof}
	Let $R$ be $7$-potent.
	By Lemma \ref{Lem:buckley}, $x^3\in Z(R)$ for all $x\in R$. Thus $(x^2+x)^3 - x^6 - x^3 = 3(x^5 + x^4)\in Z(R)$. Multiplying by $x^3$, we get $3(x^2 + x)\in Z(R)$. It follows that 
	\[
	3((x+y)^2 + x + y) - 3(x^2 + x) - 3(y^2 + y) = 3(xy + yx) \in Z(R)\,.
	\]
	By Lemma \ref{Lem:ab+ba} (with $k=3$), $3x^2\in Z(R)$. Thus $3x = 3(x^2+x) - 3x^2\in Z(R)$ for all $x\in R$.
	
	Next, we have
	\begin{align*}
		x^2 + x &= (x^2 + x)^7 \\
		&= x^{14} + 7x^{13} + 21x^{12} + 35x^{11} + 35x^{10} + 21x^9 + 7x^8 + x^7 \\
		&= x^2 + 7x + 21x^6 + 35x^5 + 35x^4 + 21x^3 + 7x^2 + x\,.
	\end{align*}
	Cancelling $x^2 + x$ from both sides and separating terms, we get
	\[
	0 = 3 (2 x + 7 x^6 + 11 x^5 + 11 x^4 + 7 x^3 + 2 x^2) + x + 2 x^5 + 2 x^4 + x^2\,.
	\]
	Thus $2 x^5 + 2 x^4 + x^2 + x\in Z(R)$. Replacing $x$ in this last expression with $-x$ and adding gives $4x^4 + 2x^2\in Z(R)$. 
	Thus $x^4 - x^2 = 4x^4 + 2x^2 - 3(x^4 + x^2)\in Z(R)$
	
	Since both $x^3$ and $x^4 - x^2$ are central, so is $(x^4 - x^2)^2 + 2(x^3)^2 = x^8 + x^4 = x^2 + x^4$. Therefore $2x^2 = (x^2 + x^4) - (x^4 - x^2)\in Z(R)$, and so $x^2 = 3x^2 - 2x^2\in Z(R)$. 
	
	Finally, $x = x^7 = x^2\cdot x^2\cdot x^3\in Z(R)$ for all $x\in R$, and we are finished.
\end{proof}

Many other cases of $x^n = x$ for small $n$ can be handled by these methods, the more difficult cases being when $n$ is odd (although Lemma \ref{Lem:buckley} certainly helps) or a power of $2$. We believe our proofs are close to being best possible (in some sense) and would be pleased to hear from any reader who can shorten or improve an any of them.

We conclude this section with one more example, covering the cases $n = 6$, $10$, $18$, $34$, $\ldots$ in a single theorem.

\begin{theorem}\label{Thm:xgen=x}
	Let $j\geq 1$ be an integer and assume that every $2^{j+1}$-potent ring is commutative. Let $i > j$ be an integer and let $R$ be a $2^i + 2^j$-potent ring. Then $R$ is commutative.
\end{theorem}
\begin{proof}
	We have $2x=0$ for all $x\in R$ by Lemma \ref{Lem:char}. Now
	\begin{alignat*}{3}
		x^2 + x &= (x^2 + x)^{2^i + 2^j}
		&&= (x^2+x)^{2^i} (x^2 + x)^{2^j} \\
		&= (x^{2^{i+1}} + x^{2^i})(x^{2^{j+1}}+x^{2^j})
		&&= x^{2(2^i + 2^j)} + x^{2^i + 2^{j+1}} + x^{2^{i+1} + 2^j} + x^{2^i + 2^j}\\
		&= x^2 + x^{2^i + 2^j + 2^j} + x^{2^i + 2^i + 2^j} + x
		&& = x^2 + x^{2^j + 1} + x^{2^i + 1} + x\,,
	\end{alignat*}
	using Lemma \ref{Lem:binom} in the third equality. Cancelling, we have
	$x^{2^j + 1} + x^{2^i + 1} = 0$, that is, $x^{2^i + 1} = x^{2^j + 1}$.
	Finally, we have 
	\[
	x = x^{2^i + 2^j} = x^{2^i + 1} x^{2^j - 1} = x^{2^j + 1} x^{2^j - 1}
	= x^{2^{j+1}}
	\]
	for all $x\in R$. By assumption, $R$ is commutative.
\end{proof}

\begin{corollary}\label{Cor:x6or10=x}
	Let $R$ be an $(2^i + 2)$-potent ring where $i > 1$. Then $R$ is commutative.
\end{corollary}
\begin{proof}
	Take $j=1$ in Theorem \ref{Thm:xgen=x} and apply Theorem \ref{Thm:x4=x}.
\end{proof}


\section{Herstein's Theorem for $n=2,4,8$}

In a ring $R$, we have already used ring commutators $[x,y] = xy - yx$ and will do so quite heavily in this section. Among the properties satisfied by commutators, we will need the following:
\[
[x,x] = 0\,,\qquad [x,y] = -[y,x]\,,\qquad [x,y + z] = [x,y] + [x,z]
\]
for all $x,y,z\in R$. To improve readability, it is useful to introduce, for each $x\in R$, the mapping $\ad{x}:R\to R$ defined by $\ad{x}(y)=[x,y]$ for all $y\in R$.
Properties satisfied by this mapping include:
\begin{align*}
	\ad{x}(y+z) &= \ad{x}(y) + \ad{x}(z)\,, \\
	\ad{x+y} &= \ad{x} + \ad{y}\,, \\	
	\ad{[x,y]} &= \ad{x}\ad{y} - \ad{y}\ad{x}\,.
\end{align*}
for all $x,y,z\in R$.
The first identity says that $\ad{x}$ is an endomorphism of the underlying abelian group $(R,+)$. The second says that $\mathrm{ad}\colon (R,+)\to \End{R,+}$ is a homomorphism of abelian groups. The third is where the symbol ``$\mathrm{ad}$'' comes from; the identity says that $\mathrm{ad}\colon (R,[\cdot,\cdot])\to \End{R,+}$ is the \emph{adjoint} representation of the Lie ring associated to $R$. However, for us, $\mathrm{ad}$ is just a notational shorthand; nothing about a ring's Lie ring plays a role outside of just using the identities above in calculations.

As discussed in {\S}1, we are interested in rings $R$ satisfying $[x^n - x,y] = 0$ for some integer $n > 1$ and all $x,y\in R$. We are specifically interested in $n=2,4$ or $8$, so we will start with the broader assumption that 
\[
[p(x) - x,y] = 0\quad\text{for all } x,y\in R \tag{E}
\]
where $p(t)$ is a polynomial with integer coefficients such that $p(t) = q(t^2)t^2$, where $q(t)\in \mathbb{Z}[t]$.

\begin{lemma}\label{Lem:even}
	Let $R$ be a ring satisfying (E). Then $2\,\ad{x}=0$ for all $x\in R$.
\end{lemma}
\begin{proof}
	For all $x\in R$, $-\ad{x} = \ad{-x} = \ad{p(-x)} = \ad{p(x)} = \ad{x}$.
\end{proof}

Roughly speaking, the conclusion of Lemma \ref{Lem:even} is, in the present setting, what replaces $R$ having characteristic $2$ in the $n$-potent setting for even $n$. The conclusion can be stated in various equivalent ways, such as $[x,y] = [y,x]$ for all $x,y\in R$.

\begin{lemma}\label{Lem:ad_squares}
	Let $R$ be a ring such that $2\,\ad{x}=0$ for all $x\in R$. Then for each nonnegative integer $n$,
	\[
	\ad{x}^{2^n} = \ad{x^{2^n}}\,.
	\]
\end{lemma}
\begin{proof}
	For $n=0$, there is nothing to prove. Assume the goal holds for some $n\geq 0$.  For $x,y\in R$, set $u = x^{2^n}$, so that $\ad{u} = \ad{x}^{2^n}$. Then
	\begin{align*}
		\ad{u^2}(y) &= [u^2,y] = u^2 y - y u^2 \\
		&= u^2 y - uyu + uyu - yu^2 \\
		&= [u,uy] + [u,yu] \\
		&= \ad{u}(uy) + \ad{u}(yu) \\
		&= \ad{u}(uy) - \ad{u}(yu) \\
		&= \ad{u}([u,y]) \\
		&= \ad{u}^2(y)\,.
	\end{align*}
	Thus
	\[
	\ad{x^{2^{n+1}}}(y) = \ad{u^2} = \ad{u}^2 = (\ad{x}^{2^n})^2(y) = \ad{x}^{2^{n+1}}(y)\,.
	\]
	By induction, we have the desired result.
\end{proof}

\begin{lemma}\label{Lem:NL}
	Let $R$ be a ring satisfying (E). Then for all $x,y\in R$,
	\[
	[x,[x,y]] = 0 \implies [x,y] = 0\,,
	\]
	or equivalently,
	\[
	\ad{x}^2(y) = 0 \implies \ad{x}(y) = 0\,.
	\]
\end{lemma}
\begin{proof}
	Assume $[x,[x,y]] = \ad{x}^2(y) = 0$ for some $x,y\in R$. By Lemmas \ref{Lem:even} and \ref{Lem:ad_squares}, $\ad{x} = \ad{p(x)} = p(\ad{x})$. Since $p(t) = q(t^2)t^2$ for some $q(t)\in \mathbb{Z}[t]$, it follows that $\ad{x}(y) = q(\ad{x}^2)\ad{x}^2(y) = 0$.
\end{proof}

\begin{lemma}\label{Lem:idem}
	Let $R$ be a ring satisfying (E). Then for all $x,y\in R$,
	\[
	[x,[x,y]] = [x,y]\implies [x,y] = 0\,,
	\]
	or equivalently,
	\[
	\ad{x}^2(y) = \ad{x}(y)\implies \ad{x}(y) = 0\,.
	\]
\end{lemma}
\begin{proof}
	Let $x,y\in R$ satisfy $[x,[x,y]] = [x,y]$. Then 
	\[
	[[x,y],[[x,y],x]] = -[[x,y],[x,[x,y]]] = -[[x,y],[x,y]] = 0\,.
	\]
	By Lemma \ref{Lem:NL}, $[[x,y],x] = 0$. Thus $[x,y] = [x,[x,y]] = 0$.
\end{proof}

\begin{theorem}\label{Thm:sq}
	Let $R$ be a ring satisfying $[x^2 - x,y] = 0$ for all $x,y\in R$. Then $R$ is commutative.
\end{theorem}
\begin{proof}
	Since (E) holds, Lemmas \ref{Lem:even} and \ref{Lem:ad_squares} give
	$\ad{x}(y) = \ad{x^2}(y) = \ad{x}^2(y)$ for all $x,y\in R$. By Lemma \ref{Lem:idem}, $\ad{x}(y) = 0$ for all $x,y\in R$, that is, $R$ is commutative. (See also Lemma \ref{Lem:x2+x}.)
\end{proof}

\begin{lemma}\label{Lem:equiv}
	Let $R$ be a ring, let $n$ be a positive integer and assume 
	$[x^{2^n} - x,y] = 0$ for all $x,y\in R$. Then 
	$[x^{2^{n-1}} + \cdots + x^2 + x,y] = 0$ for all $x,y\in R$.
\end{lemma}
\begin{proof}
	For all $x\in R$, $\ad{x}^{2^n} = \ad{x^{2^n}} = \ad{x}$, using Lemmas \ref{Lem:even} and \ref{Lem:ad_squares}. Now for all $x\in R$ (and using Lemma \ref{Lem:even}),
	\begin{align*}
		\ad{x^{2^{n-1}} + \cdots + x^2 + x}^2 &= (\ad{x^{2^{n-1}}} + \cdots + \ad{x^2} + \ad{x})^2 \\
		&= (\ad{x}^{2^{n-1}} + \cdots + \ad{x}^2 + \ad{x})^2 \\
		&= \ad{x}^{2^n} + \ad{x}^{2^{n-1}} + \cdots + \ad{x}^4 + \ad{x}^2 \\
		&= \ad{x} + \ad{x}^{2^{n-1}} + \cdots + \ad{x}^4 + \ad{x}^2 \\
		&= \ad{x^{2^{n-1}}} + \cdots + \ad{x^2} + \ad{x} \\
		&= \ad{x^{2^{n-1}} + \cdots + x^2 + x}
	\end{align*}
	By Lemma \ref{Lem:idem}, $\ad{x^{2^{n-1}} + \cdots + x^2 + x} = 0$. This proves the desired result.
\end{proof}

\begin{theorem}\label{Thm:fourth}
	Let $R$ be a ring satisfying $[x^4 - x,y] = 0$ for all $x,y\in R$. Then $R$ is commutative.
\end{theorem}
\begin{proof}
	This follows from taking $n=2$ in Lemma \ref{Lem:equiv} and using Theorem \ref{Thm:sq}.
\end{proof}

We conclude with the result that began this whole endeavor. We are well aware that of all the proofs in this paper, this is the one that most seems like it was generated by an automated deduction tool. We have simplified the proof to the point where it is possible to follow each individual step, but we would certainly agree that it is difficult to see how a human would have found the proof.

\begin{theorem}\label{Thm:exp8}
	Let $R$ be a ring satisfying $[x^8 - x,y] = 0$ for all $x,y\in R$. Then
	$R$ is commutative.
\end{theorem}
\begin{proof}
	We will freely use Lemma \ref{Lem:even} without explicit reference.
	By Lemma \ref{Lem:equiv}, $[x^4 + x^2 + x,y] = 0$ for all $x,y\in R$, that is, $\ad{x}^4 + \ad{x}^2 + \ad{x} = 0$. 
	
	First, for all $u,v\in R$,
	\begin{align*}
		(\ad{u}^2 + \ad{v}^2)([u,v]) &= (\ad{u}^2 + \ad{[u,v]} + \ad{v}^2)([u,v]) \\
		&= (\ad{u}^2 + \ad{u}\ad{v} + \ad{v}\ad{u} + \ad{v}^2)([u,v]) \\
		&= (\ad{u} + \ad{v})^2 ([u,v]) \\
		&= \ad{u+v}^2([u,v]) \\
		&= \ad{u+v}^2([u+v,v]) \\
		&= \ad{u+v}^3(v)\,.
	\end{align*}
	Thus
	\begin{align*}
		\ad{u+v}(\ad{u}^2 + \ad{v}^2)([u,v]) &= \ad{u+v}^4(v) \\
		&= \ad{u+v}^2(v) + \ad{u+v}(v) \\
		&= \ad{u+v}([u+v,v] + v) \\
		&= \ad{u+v}(v + [u,v])\,.
	\end{align*}
	Rearranging this, we have $\ad{u+v}\ad{v}^2([u,v]) = \ad{u+v}(v + [u,v] + \ad{u}^2([u,v]))$, that is,
	\[
	\ad{u+v}\ad{v}^3(u) = \ad{u+v}(v + \ad{u}(v) + \ad{u}^3(v))\,.
	\]
	In this last equation, replace $v$ with $[u,v]$ to get
	\begin{equation}\label{Eqn:exp8-1}
		[u + [u,v], \ad{[u,v]}^3(u)] = [u + [u,v],\ad{u}(v) + \ad{u}^2(v) + \ad{u}^4{v}] = 0\,.
	\end{equation}
	
	Now let $u = [x^2,y^2 + y]$ and $v = y^2 + y$. Then 
	\begin{align*}
		[u,v] &= \ad{y^2 + y}^2(x^2) \\
		&= (\ad{y}^2+\ad{y})^2(x^2) \\
		&= (\ad{y}^4 + \ad{y}^2)(x^2) \\
		&= \ad{y}(x^2) \\
		&= [x^2,y]\,,
	\end{align*}
	and $u + [u,v] = [x^2,y^2]$. Plugging all this into \eqref{Eqn:exp8-1}, we have
	\begin{equation}\label{Eqn:exp8-2}
		[[x^2,y^2], \ad{[x^2,y]}^3([x^2,y^2 + y])] = 0\,.
	\end{equation}
	Now
	\begin{align*}
		\ad{[x^2,y]}^3([x^2,y^2 + y]) &= \ad{[x^2,y]}^2([[x^2,y],[x^2,y^2] + [x^2,y]]) \\
		&= \ad{[x^2,y]}^3([x^2,y^2]) \\
		&= \ad{[x^2,y]}^3\ad{y}^2(x^2) \\
		&= \ad{[x^2,y]}^3\ad{y}([x^2,y]) \\
		&= \ad{[x^2,y]}^3[[x^2,y],y]) \\
		&= \ad{[x^2,y]}^4(y) \\
		&= \ad{[x^2,y]}^2(y) + \ad{[x^2,y]}(y) \\
		&= \ad{[x^2,y]}\ad{y}^2(x^2) + \ad{y}^2(x^2) \\
		&= [[x^2,y],[x^2,y^2]] + [x^2,y^2]\,.
	\end{align*}
	Thus the left side of \eqref{Eqn:exp8-2} simplifies to
	\[
	[[x^2,y^2],[[x^2,y],[x^2,y^2]] + [x^2,y^2]] = [[x^2,y^2],[[x^2,y],[x^2,y^2]] = \ad{[x^2,y^2]}^2([x^2,y])\,.
	\]
	
	Therefore \eqref{Eqn:exp8-2} reduces to
	\[
	\ad{[x^2,y^2]}^2([x^2,y]) = 0
	\]
	for all $x,y\in R$. By Lemma \ref{Lem:NL},
	\[
	0 = \ad{[x^2,y^2]}([x^2,y]) = \ad{[x^2,y]}\ad{y}^2(x^2) = \ad{[x^2,y]}^2(y)\,.
	\]
	By Lemma \ref{Lem:NL} again, $\ad{[x^2,y]}(y) = 0$, 
	that is, $\ad{y}^2(x^2) = 0$. By Lemma \ref{Lem:NL} again, $\ad{y}(x^2) = [x^2,y] = \ad{x}^2(y) = 0$.
	Applying Lemma \ref{Lem:NL} one last time, $[x,y] = 0$. Therefore $R$ is commutative.
\end{proof}

\bigskip
\noindent
{\small
\authorbio
}

\end{document}